\documentclass[11pt]{article}
\usepackage[english]{babel}
\usepackage[T1]{fontenc}
\usepackage[latin1]{inputenc}

\usepackage{amsfonts,amssymb,amsmath,amsthm, graphics, setspace}
\usepackage{bbm}
\usepackage{xcolor}
\usepackage{graphics}
\usepackage{graphicx}
\usepackage{mathtools}
\usepackage{hyperref} 
\usepackage{soul}
\usepackage{ulem}

\usepackage{authblk}

\newtheorem{theorem}{Theorem}[section]

\newtheorem{corollary}[theorem]{Corollary}

\newtheorem{proposition}[theorem]{Proposition}
\newtheorem{lemma}[theorem]{Lemma}

\theoremstyle{remark}

\numberwithin{equation}{section}

\title{Genetic contribution of advantaged ancestors in the biparental Moran model - finite selection}
\author[1]{Camille Coron}
\author[2]{Yves Le Jan}
\affil[1]{Universit\'e Paris-Saclay, AgroParisTech, INRAE, UMR MIA Paris-Saclay, 91120, Palaiseau, France}
\affil[2]{Universit\'e Paris-Saclay, CNRS, Laboratoire de math\'ematiques d'Orsay, 91405, Orsay, France.}
\date{}

\usepackage{color}
\usepackage{amsmath}
\usepackage{amsthm}

\begin{document}
\maketitle

\begin{abstract}
We study a population of $N$ individuals evolving according to a biparental Moran model with two types, one being advantaged compared to the other. The advantage is conferred by a Mendelian mutation, which reduces the death probability of individuals carrying it. We assume that a proportion $a$ of individuals initially carry this mutation, which therefore eventually gets fixed with high probability. After a long time, we sample a gene uniformly from the population, at a new locus, independent of the locus under selection, and calculate the probability that this gene originated from one of the initially advantaged individuals, when the population size is large. Our theorem provides quantitative insights, such as the observation that under strong viability selection, if only $1\%$ of the individuals are initially advantaged, up to $19\%$ of the population's genome will originate from them after a long time.\\

\end{abstract}


\section{Motivation, model and main results}

\subsection{Motivation}

This article investigates how carrying a beneficial genetic mutation influences the quantity of transmitted genome, in a population of sexually reproducing hermaphroditic individuals. Building on previous works (\cite{geneal} and \cite{CoronLeJan22}), we consider a population, represented by a biparental Moran model, of haploid individuals whose genome is a random and balanced mixture of their two parents' genomes. Initially, a proportion $a$ of individuals carry a mutation that increases their life expectancy and thus their genetic contribution to the population. Our goal is to quantify the effect of selection strength on the genetic composition of the population in the limit of large population size and more specifically to determine the impact of the advantage conferred by the mutation on the genetic contribution of ancestors. Our approach could be used for instance to estimate the strength of selection in a population composed of two types of individuals (for instance coming from two different source populations), using admixture data, as is the object of the works \cite{VerduRosenberg2011,FortesLimaetal2021}.

Our paper focuses on two major phenomena in population genetics: on the one hand, the impact of selection on genome structure, which has been the subject of numerous studies, and on the other hand, the biparental transmission of the genome, which has been relatively little studied so far. Biparental genealogies and genetics have yet received some interest, notably in \cite{Chang1999,Derrida2000,GravelSteel2015}, in which time to more recent common ancestors and ancestors' weights are investigated for the Wright-Fisher biparental model. In \cite{geneal}, we studied the asymptotic law of the  contribution of an ancestor, to the genome of the present time population, within a neutral genetic framework. In \cite{Lambertetal2018}, the authors studied the asymptotic structure of the genome of individuals in a population experiencing biparental haploid reproduction, while in \cite{PfaffelhuberWakolbinger2023} the authors studied the biparental transmission of transposable elements in diploid populations. The monoparental Moran model with selection at birth has received some interest, notably in \cite{EtheridgeGriffiths} that studies its dual coalescent and \cite{KluthBaake} that studies alleles fixation probabilities and ancestral lines. The impact of selection on genetic transmission has also been studied using the concept of ancestral selection graph, introduced in  \cite{KroneNeuhauser1997}, that notably led to the articles \cite{Fernhead2002, Taylor2007, Kluthetal2013, Lenzetal2015, BaakeLenzWakolbinger2016, Cordero2017} dealing with ancestors type distribution in populations modeled by monoparental bi-type Wright-Fisher or Moran models with mutation and selection, and using the lookdown construction. Finally, a few articles combine selection and biparental pedigree. The articles \cite{MatsenEvans2008} and \cite{BartonEtheridge2011} study the link between pedigree, individual reproductive success and genetic contribution while in \cite{CoronLeJan22} we studied the expectation of the asymptotic contribution to the genome of a biparental population, of a mutant subject to infinite selection.

Studying the dynamics of genome transmission during the fixation period of an advantageous allele is linked to studying genetic hitchhiking, which is done notably in \cite{Durrett2004,Etheridgeetal2006} for haploid individuals and in \cite{Stephanetal2006} for diploid individuals. Genetic hitchhiking is a phenomenon through which the dynamics of the frequency of an allele is influenced by the dynamics of genes under selection, that are close to this first allele, or at least on the same DNA strand. This phenomenon intrinsically relies on the spatial structure of the genome, and on the fact that two pieces of genomes that are close to each other on the genome have a large probability to be transmitted together during reproduction. In this article, we study the transmission of the genome of advantaged individuals, while assuming that the genome is composed of infinitely many independent loci (which is obviously a strong simplification). In some sense, we study the weakest form of genetic hitchhiking, in which pieces of genome are linked only through the fact that they are present in the same individual and are transmitted through the same pedigree, that is itself shaped by the transmission of advantageous alleles. The form and role of the pedigree on unlinked neutral loci during a selective sweep is also investigated using simulations in \cite{WakeleyKingWilton2016}, for populations of diploid individuals.

Note. A manuscript bearing a similar title has been deposited on arXiv and is not intended for publication. It contains the proof of a weaker convergence statement, based on properties of the discrete model, some of which have not been included in the present work.

\subsection{Model}\label{sec:model}

Following the framework of previous papers (\cite{geneal} and \cite{CoronLeJan22}), we model the dynamics of a population of $N$ haploid individuals using a biparental two-type discrete time Moran model with viability selection and no mutation. Time will be represented by the set $\mathbb{Z}_+=\{0,1,...\}$ of non-negative integers.

Specifically, we consider a population of fixed size 
$N$ in which a mutation at a given locus confers an advantage to individuals carrying it. We assume that selection affects only the death rate of individuals: advantaged individuals have a death weight of $1$, while non-advantaged individuals have a death weight of $1+s$, with $s>0$. At each discrete time step, two individuals are chosen uniformly at random with replacement to be parents and produce one offspring, which replaces a third individual chosen with probability proportional to its death weight (and independently from the two parents, which means that a child can replace its own parent, and also that both parents can be the same individual). This results in a death probability at each time step that is $1+s$ times higher for a non-advantaged individual than for an advantaged individuals, leading to an increased life expectancy and mean offspring number for the latter. The selection parameter $s$ will not depend on $N$, as is often the case in population genetics literature, and in particular we will not explore diffusive scalings introduced in \cite{Kimura1964}.

Genetic transmission follows Mendelian rules, in the sense that at a given locus of the genome, each individual inherits an allele, chosen uniformly at random among the two alleles of its parents. In particular, the transmission of advantage to offspring is characterized by Mendelian transmission at one locus. We refer to the parent that transmits a copy of its gene at the locus under selection as the "mother" and the other parent as the "father". An individual is therefore advantaged if and only if its mother is advantaged. The limiting case where $s$ is infinite (i.e., only non-advantaged individuals can die) has been studied in a previous paper \cite{CoronLeJan22}, and comparisons to this work will be provided later on. As in the two previous works \cite{geneal} and \cite{CoronLeJan22}, recombination is not considered. We indeed focus on the probability for a gene sampled at present time, to originate from a given ancestor, which can be seen as the proportion of genome transmitted by this ancestor, if the genome is seen as a set of infinitely many independent loci. 

Let us denote by $I=\{1,2,...,N\}$ the sites in which individuals live (which is simply a way to number individuals at each time step), and denote by $(\mu_n,\pi_n,\kappa_n)\in I^3$ the respective positions of the mother, father, and dying individual at time step $n\in\mathbb{Z}_+$. As in \cite{geneal}, the reproduction dynamics described above defines an oriented random graph on $I\times\mathbb{Z}_+$ (as represented in Figure \ref{FigPedigree}), denoted by $G$, representing the pedigree of the population, such that between time $n$ and time $n+1$, two arrows are drawn from $(\kappa_n,n+1)$ to $(\pi_n,n)$ and $(\mu_n,n)$ respectively and $N-1$ arrows are drawn from $(i,n+1)$ to $(i,n)$ for each $i\in I\setminus\{\kappa_n\}$. Note that individuals are now also characterized by their advantage: advantaged individuals are represented in red in Figure \ref{FigPedigree}. We denote by $\mathcal{Y}_n\subset \{1,...,N\}$ the set of advantaged individuals at time $n$. We also define by $(\mathcal{F}_n,n\in\mathbb{Z}_+)$ the filtration associated to the whole population dynamics, i.e. $\mathcal{F}_n=\sigma(\mathcal{Y}_0,\{\mu_l,\pi_l,\kappa_l, l<n\})$ for any $n\in\mathbb{Z}_+$. The filtration $(\mathcal{F}_n,n\in\mathbb{Z}_+)$ notably includes the filtration associated to the stochastic process $(\mathcal{Y}_n)_{n\in\mathbb{Z}_+}$. Note that the pedigree is characterized only by the sequence $((\{\mu_n,\pi_n\},\kappa_n))_{n\in\mathbb{Z}_+}$, which is not a Markov chain, due to selection. We denote by $\{\mathcal{G}_n,n\in\mathbb{Z}_+\}$ with $\mathcal{G}_n=\sigma(\{\{\mu_l,\pi_l\},\kappa_l,l<n\})$ its associated filtration.

\begin{figure}[ht]
\begin{center}\includegraphics[scale=0.5]{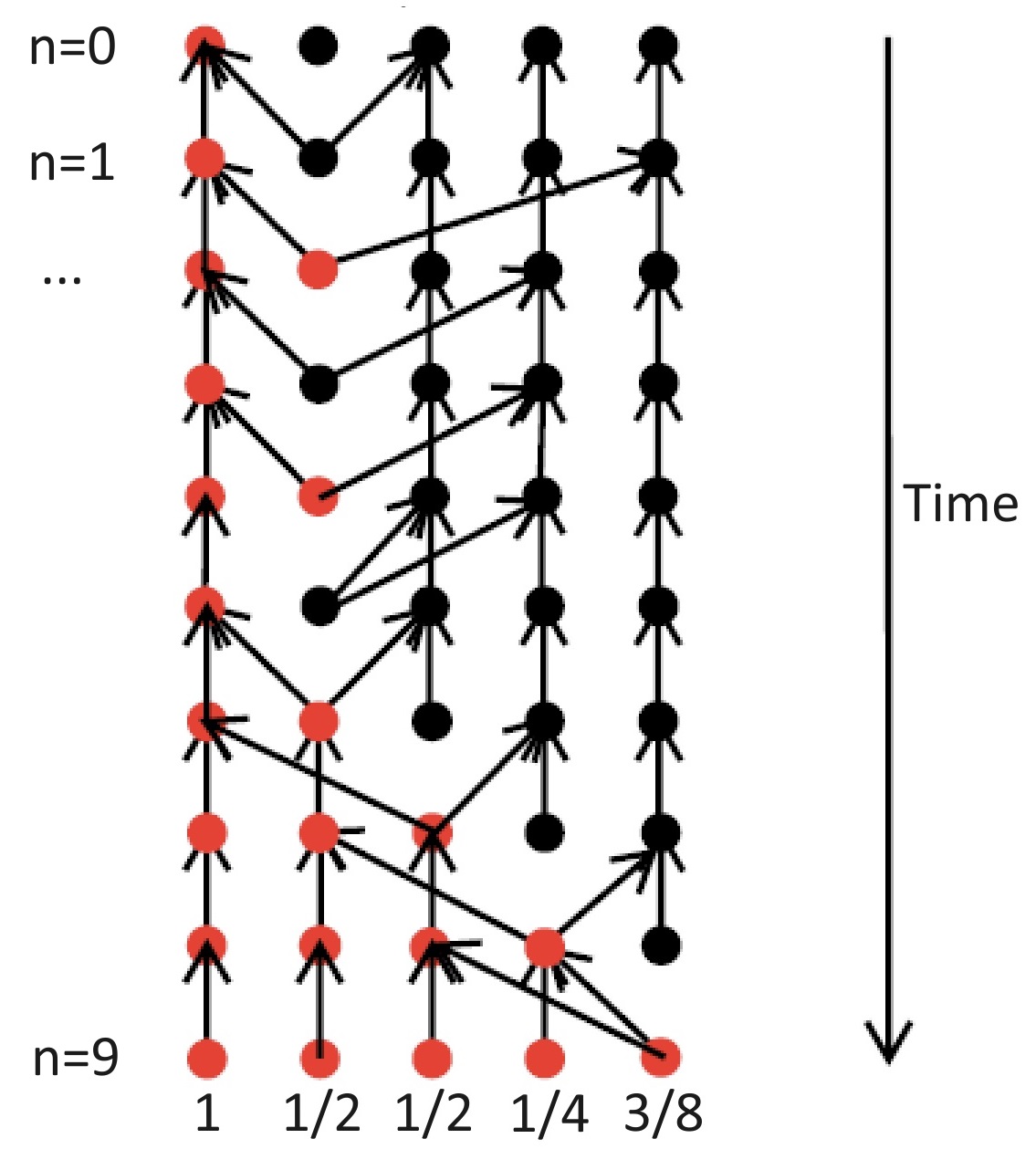}
\end{center} 
\caption{This figure represents simultaneously the pedigree $G$ and the set of advantaged individuals, at $10$ time steps, in a population of $5$ individuals. The time orientation is from past to future and arrows materialize gene flow between individuals, represented backwards in time. Advantaged individuals are represented in red. Numbers at the bottom gives the probability that a gene sampled in each of the individuals come from the initially advantaged individual. In this example the genetic weight of the initially advantaged individual (i.e. $N$ times the probability, given the pedigree, that a gene sampled uniformly at time $n=9$ comes from this individual at time $n=0$) is equal to $21/8=1+1/2+1/2+1/4+3/8$.\label{FigPedigree}}
\end{figure}

\subsection{Ancestors genetic weights}

To investigate the impact of selection on the genetic composition of the population, we consider a second locus that is far enough from the locus under mutation so that the genome is assumed to be transmitted independently at these two loci. We then sample an individual in the population at time $n$ and consider the genealogy of its gene at this second locus, i.e. the sequence of ancestral individuals through which this gene was transmitted, at all times $0\leq n-k\leq n$. This genealogy, denoted by $(X^{(n)}_k, n-k)_{0\leq k\leq n}$, is a symmetric random walk on the pedigree $G$. A particularly interesting element is the number $X_n^{(n)}$ that gives the position of the initial ancestor of the sampled gene.

The key object in this model, as introduced in \cite{geneal} is therefore the sequence of random variables $((W_n(i,j))_{1\leq i,j\leq N})_{n\geq 0}$ defined by

\begin{equation} \label{eq:defA} W_n(i,j)=\mathbb{P}(X^{(n)}_n=j|X^{(n)}_0=i,\mathcal{G}_n).\end{equation}

In words, the quantity $W_n(i,j)$ is the probability, given the pedigree before time $n$, $\mathcal{G}_n$, that any gene of individual $i$ living at time $n$ comes from the ancestor $j$, living at time $0$. It is a deterministic function of the random graph $G$ between time $0$ and time $n$ and does not depend on the advantage status of individuals in the pedigree (although the pedigree itself depends on this status). More formally, $W_n(i,j)=\mathbb{P}(X^{(n)}_n=j|X^{(n)}_0=i,\mathcal{F}_n)$.

If the genome is very large and the evolution of distant genes are sufficiently decorrelated, we can expect this quantity to be close to the proportion of genes of individual $i$ that come from individual $j$. For this reason we refer to this quantity as the \textit{genetic weight at time $n$ of ancestor $j$ in the genome of individual $i$} (see Figure \ref{FigPedigree} for an illustration). We define the genetic weight of ancestor $j$ at time $n$ as the sum $\sum_{i=1}^NW_n(i,j)$.

\subsection{Main result}\label{sec:IntroResult}

Recall that $\mathcal{Y}_n\subset I$ denotes the set of advantaged individuals at time $n$, and let $Y_n$ be its cardinality, i.e. the number of advantaged individuals at time $n$. Our goal is to investigate the impact of selection on the weight of ancestors, and notably to study the probability that a gene sampled in the population at present time originated from an advantaged individual. 

To this aim we introduce the two following quantities: 
\begin{align*}
U_n&=\sum_{l\in\mathcal{Y}_n}\sum_{l'\in\mathcal{Y}_0} W_n(l,l'), \quad\text{and}\quad
V_n=\sum_{l\notin\mathcal{Y}_n}\sum_{l'\in\mathcal{Y}_0} W_n(l,l').
\end{align*}

The quantity $U_n$ (resp. $V_n$) will be called the total  genetic weight of initially advantaged ancestors among advantaged (resp. disadvantaged) individuals, at time $n$ ; it gives the quantity of genome coming from the initially advantaged ancestors and carried by advantaged (resp. disadvantaged) individuals, assuming that each individual carries a quantity of genome equal to $1$.
The quantity $U_n\in[0,N]$ (resp. $V_n\in[0,N]$) is equal to $Y_n$ (resp. $N-Y_n$) times the probability, knowing the pedigree, that a gene sampled uniformly among advantaged (resp. non-advantaged) individuals at time $n$ originates from an initially advantaged individual. More formally, as long as $Y_n\notin\{0,N\}$, if we denote by $\mathcal{U}(A)$ the uniform law on a discrete set $A$,

$$\frac{U_n}{Y_n}=\mathbb{P}(X_n^{(n)}\in \mathcal{Y}_0| X^{(n)}_0\sim\mathcal{U}(\mathcal{Y}_n),\mathcal{G}_n),$$ and $$\frac{V_n}{N-Y_n}=\mathbb{P}(X_n^{(n)}\in\mathcal{Y}_0| X^{(n)}_0\sim\mathcal{U}(I\setminus\mathcal{Y}_n),\mathcal{G}_n).$$

We denote by $T_k$ the first hitting time of any integer $k$ lying between $0$ and $N$, by the Markov chain $(Y_n)_{n\in\mathbb{N}}$. Then the two quantities $$\frac{U_{T_N}}{N}\mathbf{1}_{\{T_N<\infty\}}=\frac{U_{T_N}}{Y_{T_N}}\mathbf{1}_{\{T_N<\infty\}}\quad\text{and}\quad\frac{V_{T_0}}{N}\mathbf{1}_{\{T_0<\infty\}}=\frac{V_{T_0}}{N-Y_{T_0}}\mathbf{1}_{\{T_0<\infty\}}$$
represent the proportion of genome of the population coming from initially advantaged individuals once the mutation is fixed or has disappeared, respectively. They indeed give the probability that a gene sampled uniformly from the population, once the latter has become monomorphic, originates from an advantaged individual (see Figure \ref{FigPedigree} for an example starting with one advantaged individual, and fixation of this type). 
 
After the time $\inf(T_0,T_N)$, the population continues to evolve, according to the neutral model studied in \cite{geneal}. Note finally that although $W_n$ is a $\mathcal{G}_n$-measurable function, the genetic weights $U_n$ and $V_n$ are not, since the pedigree alone does not give the advantage status of individuals. The precise dynamics of all these stochastic processes will be given in the next section, but our brief introduction so far suffices to state our main result:

\begin{theorem}\label{thmCvceDeterministe} Let $Z_n=\left(\frac{Y_n}{N},\frac{U_n}{N},\frac{V_n}{N}\right)_{n\in\mathbb{N}}$. Let $a\in(0,1)$. If $Y_0=\lfloor aN\rfloor$, then for any $c\in\mathbb{R}_+$,
\begin{align}\sup_{0\leq t\leq c}\|Z_{\lfloor Nt\rfloor}-z_t\|\underset{N\rightarrow\infty}{\longrightarrow}0\end{align} 
in probability, where $(z_t)_{t\geq0}=(y_t,u_t,v_t)_{t\geq0}$ satisfies
$$\left\{\begin{aligned}
      y_t&= F^{-1}\left(\frac{a^{1+s}}{1-a}\exp(st)\right) \qquad\text{where  $F:x\mapsto\frac{x^{1+s}}{1-x}$ maps $[0,1)$ onto $[0,\infty)$}
      \\
      u_t&= y_t\frac{a^{\frac{1+s}{2s}}}{(1-a)^{\frac{1}{2s}}}\left[\frac{(1-y_t)^{\frac{1}{2s}}}{y_t^{\frac{1+s}{2s}}}+\int_a^{y_t}\frac{(1-x)^{\frac{1}{2s}}}{x^{\frac{1+s}{2s}}}\left[\frac{1}{2}+\frac{1}{2s}\frac{1}{1-x}\right]dx\right] \\
      v_t&= (1-y_t)\frac{a^{\frac{1+s}{2s}}}{(1-a)^{\frac{1}{2s}}}\int_a^{y_t}\frac{(1-x)^{\frac{1}{2s}}}{x^{\frac{1+s}{2s}}}\left[\frac{1}{2}+\frac{1}{2s}\frac{1}{1-x}\right]dx. \\
    \end{aligned}\right.$$
\end{theorem}

This result gives the convergence of the trajectory of the stochastic process $(Z_{\lfloor Nt\rfloor})_{0\leq t\leq c}$ for any $c>0$, towards an explicit solution of a dynamical system. In particular, it gives the genetic contribution of the initially advantaged individuals in compact time intervals, under a large population size assumption. 
More precisely, it gives the limiting dynamics of the proportion of advantaged individuals ($Y_{\lfloor Nt\rfloor}/N$), and of the proportion of genome coming from advantaged individuals carried respectively by advantaged ($U_{\lfloor Nt\rfloor}/N$) and disadvantaged ($V_{\lfloor Nt\rfloor}/N$) individuals. In Section \ref{sec:proofs} we will see that the limiting dynamics of the proportion of advantaged individuals results from an autonomous differential equation $y'=\frac{sy(1-y)}{y+(1+s)(1-y)}$. This differential equation driven by the selection parameter $s$ can be compared to that obtained for the classical (monoparental) Moran model, for which the proportion $z$ of advantaged individuals under a similar scaling satisfies $z'=sz(1-z)$ (as can be seen in \cite{Cordero2017}, Proposition 3.1 for instance). The difference between the two equations lies in the biparental Mendelian transmission versus the clonal transmission. A simple and unsurprising comparison between the two models is that the proportion of advantaged individuals increases slower in the biparental Mendelian model, since $y+(1+s)(1-y)>1$.
Note that if the initial proportion of advantaged individuals $\frac{Y_0}{N}$ is equal to $\lfloor aN\rfloor$ then $\mathbb{P}(T_N<T_0)\rightarrow 1$ when $N$ goes to infinity (Proposition \ref{prop-Y-TN}).
The following corollary gives, under a large population approximation, what will be the genetic weight of the initially advantaged ancestors, initially assumed to be in proportion $a$, when this proportion reaches any level $b>a$. 
\begin{corollary}\label{cor_main} Let $a<b\in(0,1)$. If $Y_0=\lfloor aN\rfloor$ then 
\begin{align*}\frac{U_{T_{\lfloor bN\rfloor}}}{N}\mathbf{1}_{\{T_{\lfloor bN\rfloor\}}<\infty}\underset{N\rightarrow\infty}{\longrightarrow} b\frac{a^{\frac{1+s}{2s}}}{(1-a)^{\frac{1}{2s}}}&\left[\frac{(1-b)^{\frac{1}{2s}}}{b^{\frac{1+s}{2s}}}+\int_a^{b}\frac{(1-x)^{\frac{1}{2s}}}{x^{\frac{1+s}{2s}}}\left[\frac{1}{2}+\frac{1}{2s}\frac{1}{1-x}\right]dx\right] \end{align*} in probability.
\end{corollary}

We think that this still holds true for $b=1$, but proving it would necessitate a finer study of the variance of $U_n$ (see \cite{coronLeJan2024} for the proof of the convergence of the expectation of $\frac{U_{T_{N}}}{N}$). As an example, if $s=1$, then Corollary \ref{cor_main} tells us that if $1\%$ of the individuals are initially advantaged, then they will on average end to be responsible for approximately $5\%$ of the whole population's genome (instead of $1\%$ in the neutral case where $s=0$). Note besides that the limit given in Corollary \ref{cor_main} is non-decreasing in $s$, and one can see that it converges to $2\sqrt{a}-a$ as $s\uparrow \infty$ and $b$ converges to $1$. This result can also be retrieved from \cite{CoronLeJan22}, in which we focused on the case where $s=\infty$. In particular, for large selection strength $s$, if only $1\%$ of the individuals are initially advantaged, they will on average be responsible for approximately $19\%$ of the population's genome.
Theorem \ref{thmCvceDeterministe} and Corollary \ref{cor_main} are illustrated in Figure \ref{fig:simulations}. 
\begin{figure}
\centering
\begin{minipage}{0.5\linewidth}
\centering
  \includegraphics[height=5cm]{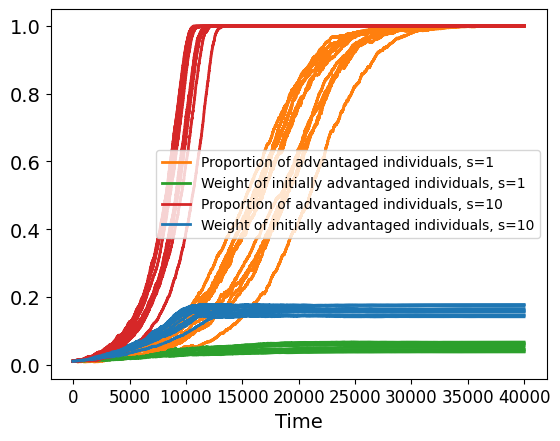}
\end{minipage}\quad
\begin{minipage}{0.46\linewidth}
\centering
    \includegraphics[height=5cm]{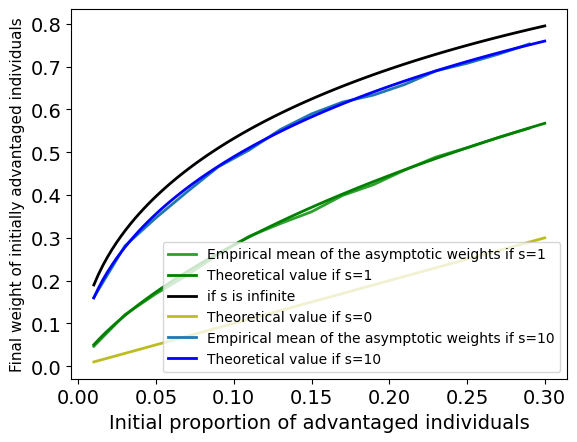}
\end{minipage}
 \caption{Left: For $N=1000$, and $a=1\%$, we plot $10$ realizations of the dynamics of the proportion of advantaged individuals (in red for $s=10$ and in orange for $s=1$) and the genetic weight of initially advantaged individuals (in blue for $s=10$ and in green for $s=1$). Right: For different values of initial proportion of advantaged individuals, we plot for $s=1$ (in green) and for $s=10$ (in blue) the mean of the weight at time $n=40000$ of the initially advantaged individuals over $10$ simulations. We also give the associated theoretical predictions, as well as the theoretical predictions for $s=0$ (in yellow) and for infinite $s$ (in black).}    \label{fig:simulations}
\end{figure}

\section{Proofs}\label{sec:proofs}
\subsection{A few Markovian properties}

We begin by presenting a series of results that shed light on various aspects of the dynamics of the process under consideration, and notably the interplay between the set and number of advantaged individuals, the pedigree, and the genetic weights of ancestors. 

\paragraph*{Number and set of advantaged individuals } We first focus on the set and number of advantaged individuals, whose dynamics happen to be particularly simple. Recall that $\mu_n\in I$ is the site of the mother (parent that transmits advantage) at time step $n$, while $\kappa_n\in I$ is the site of the individual that is killed at  step $n$. These notations combined to the modeling framework described in Section \ref{sec:model}
immediately give the following
\begin{proposition} \label{prop:Avantages}
The stochastic process $(\mathcal{Y}_n)_{n\in\mathbb{Z}_+}$ is a Markov chain, such that $$\mathcal{Y}_{n+1}=\begin{cases}
  \mathcal{Y}_n\cup\{\kappa_n\}  \qquad & \text{if $\mu_n\in\mathcal{Y}_n$ and $\kappa_n\notin\mathcal{Y}_n$} \\
  \mathcal{Y}_n\setminus\{\kappa_n\} &\text{if $\mu_n\notin\mathcal{Y}_n$ and $\kappa_n\in\mathcal{Y}_n$}\\
  \mathcal{Y}_n &\text{otherwise.}
\end{cases}$$
\end{proposition}

\begin{proof}
If the mother is chosen among advantaged individuals then by the definition of the model described in Section \ref{sec:model}, the site $\kappa_n$ is added to (or remains in) the set of advantaged individuals. If the mother is chosen among non advantaged individuals then the site $\kappa_n$ is removed from the set of advantaged individuals, if it was present in this set.
\end{proof}

As a consequence,

\begin{proposition}\label{prop:Y}
The stochastic process  $(Y_n)_{n\in\mathbb{Z}_+}$ is a Markov chain such that if $Y_n=k\in \{0,1,...,N\}$, then $Y_{n+1}\in\{k-1,k,k+1\}$, and \begin{align*}\mathbb{P}(Y_{n+1}=k-1|Y_n=k)&=p_k\times\frac{1}{2+s},\\\mathbb{P}(Y_{n+1}=k+1|Y_n=k)&=p_k\times\frac{1+s}{2+s},\quad\text{and}\\\mathbb{P}(Y_{n+1}=k|Y_n=k)&=1-p_k,\end{align*} where $p_k=\frac{k(N-k)}{N\left(\frac{1}{2+s}k+\frac{1+s}{2+s}(N-k)\right)}=\frac{1}{\left(\frac{1}{2+s}\frac{N}{N-k}+\frac{1+s}{2+s}\frac{N}{k}\right)}.$ This Markov chain is absorbed in $0$ and in $N$.
\end{proposition}

\begin{proof} As summarized in Proposition \ref{prop:Avantages}, the number of advantaged individuals is increased by $1$ if the mother is advantaged while the dying individual is disadvantaged, which gives that $\mathbb{P}(Y_{n+1}=k+1|Y_n=k)=k/N\times\frac{(1+s)(N-k)}{k+(1+s)(N-k)}=\frac{1+s}{2+s}\times \frac{k(N-k)}{N(\frac{1}{2+s}k+\frac{1+s}{2+s}(N-k))}$. Similarly, the number of advantaged individuals is decreased by $1$ if the mother is disadvantaged while the dying individual is advantaged, which gives that $\mathbb{P}(Y_{n+1}=k-1|Y_n=k)=(N-k)/N\times\frac{k}{k+(1+s)(N-k)}=\frac{1}{2+s}\times \frac{k(N-k)}{N(\frac{1}{2+s}k+\frac{1+s}{2+s}(N-k))}$. Finally, the number of advantaged individuals remains the same if the mother and replaced individual have the same advantage status, which gives that $\mathbb{P}(Y_{n+1}=k|Y_n=k)=\frac{k}{N}\frac{k}{k+(1+s)(N-k)}+\frac{N-k}{N}\frac{(1+s)(N-k)}{k+(1+s)(N-k)}=\frac{k^2+(1+s)(N-k)^2}{N(k+(1+s)(N-k))}=1-\frac{(2+s)k(N-k)}{N(k+(1+s)(N-k))}=1-p_k$. As $p_0=p_N=0$, this gives that the states $0$ and $N$ are absorbing.
\end{proof}

The previous result implies that the skeleton of the Markov chain $(Y_n)_{n\in\mathbb{Z}_+}$ has a particularly simple dynamics. More precisely, let $\tau_0=0$, and $H_0=Y_0=Y_{\tau_0}$. Now for any $l\in\mathbb{Z}_+^*$, if $H_l\notin\{0,N\}$ let $\tau_{l+1}=\inf\{n>\tau_l|Y_n\neq H_l\}$ and $H_{l+1}=Y_{\tau_{l+1}}$. If $Y_{\tau_l}\in\{0,N\}$ let $H_{l+1}=H_l$. Then from Proposition \ref{prop:Y}, 
\begin{proposition}\label{prop-biasedRW}
The stochastic process $(H_l)_{l\in\mathbb{Z}_+}$ is a simple random walk absorbed in $\{0,N\}$: for any $l\in\mathbb{Z}_+$ such that $H_l\notin\{0,N\}$,
\begin{align}\mathbb{P}(H_{l+1}&=H_{l}+1)=\frac{1+s}{2+s}\\\mathbb{P}(H_{l+1}&=H_{l}-1)=\frac{1}{2+s},\end{align} and if $H_l\in\{0,N\}$, $H_{l+1}=H_l$. We denote by $S_k$ the first hitting time of $k\in\{0,1,...,N\}$ by the random walk $H$.
\end{proposition}

This result allows to prove very simply the first item of the following proposition:
\begin{proposition}\label{prop-Y-TN} Let $a\in(0,1)$. If $Y_0=\lfloor aN\rfloor$, then 
\begin{description}
\item[$(i)$] The fixation probability of advantaged individuals satisfies \begin{align*}\mathbb{P}(T_N<T_0)\underset{N\rightarrow\infty}{\longrightarrow} 1. \end{align*} 
\item[$(ii)$] Let $b\in [a,1)$. There exists $C>0$ (depending on $b$) such that $\mathbb{P}(T_{\lfloor bN\rfloor}> N C)\rightarrow 0$ when $N$ goes to infinity.
\end{description}
\end{proposition}
\begin{proof} We recall that $S_k$ is the first hitting time of $k\in\{0,1,...,N\}$ by the biased random walk $(H_l)_{l\in\mathbb{Z}_+}$. Remark that $T_N<T_0$  if and only if $S_N<S_0$. Then result $(i)$ falls from Proposition \ref{prop-biasedRW} and from a classical result on "gambler's ruin"  (see for example Chap. XIV-1 of \cite{feller-vol-1}).  
Now, let us turn our attention to (ii). The proof is decomposed in two parts: firstly we prove the result for the absorbed biased random walk $(H_l)_{l\in\mathbb{Z}_+}$, i.e. we prove that for any $b\in[a,1)$ there exists a constant $K>0$ such that  $\mathbb{P}(S_{\lfloor bN\rfloor}> N K)\rightarrow 0$ when $N$ goes to infinity. Secondly we prove that the Markov chain $(Y_n)_{n\in\mathbb{Z}_+}$ cannot stay for a long time at each position, by noting that the parameters $p_k=\mathbb{P}(Y_{n+1}\neq Y_n|Y_n=k)$ are bounded from below when the Markov chain $(Y_n)_{n\in\mathbb{Z}_+}$ stays away from $\epsilon N$, which occurs with high probability. We conclude using the Bienaym\'e-Chebyshev inequality.
\\
More formally, we bound the probability of interest as follows. For any $C, K>0$,
\begin{align}\label{eq-decompoT}\mathbb{P}(T_{\lfloor bN\rfloor}>CN)&\leq \mathbb{P}(T_{\lfloor bN\rfloor}>CN,S_{\lfloor \epsilon N\rfloor}\geq S_{\lfloor bN\rfloor},S_{\lfloor bN\rfloor}\leq KN)\\&+\mathbb{P}(S_{\lfloor bN\rfloor}>KN)\\&+\mathbb{P}(S_{\lfloor \epsilon N\rfloor}<S_{\lfloor bN\rfloor})\end{align} and control the three terms of the right hand side of this equation. 
For the last one, note that for any $\epsilon<a$, $\mathbb P(S_{\lfloor bN\rfloor}<S_{\lfloor \epsilon N\rfloor})=\frac{(1+s)^{\lfloor \epsilon N\rfloor-\lfloor a N\rfloor}-1}{(1+s)^{\lfloor \epsilon N\rfloor-\lfloor bN\rfloor}-1}\rightarrow 1$ as $N$ goes to infinity.
For the second term, let us introduce a random walk $\tilde{H}$ which has the same transition probabilities and starting point $\lfloor aN\rfloor$ as the random walk $H$ but is not absorbed in $\{0,N\}$, i.e. which is such that $\mathbb{P}(\tilde{H}_{l+1}=\tilde{H}_{l}+1)=\frac{1+s}{2+s}$ and $\mathbb{P}(\tilde{H}_{l+1}=\tilde{H}_{l}+1)=\frac{1}{2+s}$ whatever value  $\tilde{H}_l$ takes in $\mathbb{Z}$. We have for any $K>0$,
\begin{equation}\label{eq-control}
\begin{aligned}\mathbb{P}(S_{\lfloor bN\rfloor}>KN)\leq\mathbb{P}(H_{KN}<\lfloor bN\rfloor)&\leq \mathbb{P}(H_{KN}=0)+\mathbb{P}(0<\tilde{H}_{KN}<\lfloor bN\rfloor))\\&\leq \mathbb{P}(H_{KN}=0)+\mathbb{P}(\tilde{H}_{KN}<bN)).\end{aligned}\end{equation}
Now recall that $H_0=Y_0=\lfloor aN\rfloor$. Therefore $$\mathbb{P}(H_{KN}=0)=\mathbb{P}(S_0\leq KN)\leq \mathbb{P}(S_0<\infty)\rightarrow 0$$ as $N$ goes to infinity from the same gambler's ruin result as before.
Besides, by the Law of large numbers, $\frac{\tilde{H}_{KN}-\lfloor aN\rfloor}{N}$ converges to  $\frac{s}{2+s} K$ when $N$ goes to infinity. Hence $\mathbb{P}(\tilde{H}_{KN}<\lfloor bN\rfloor)$ goes to $0$ when $N$ goes to infinity as soon as $K>\frac{(b-a)(2+s)}{s}$. Therefore from \eqref{eq-control}, $\mathbb{P}(S_{\lfloor bN\rfloor}>KN)$ goes to $0$ when $N$ goes to infinity if $K$ is large enough.
\\
We finally consider the first term of the right-hand side of Equation \eqref{eq-decompoT}. We know that $$\mathbb{P}(Y_{n+1}\neq Y_n|Y_n=k)=p_k= \frac{k(N-k)}{N\left(\frac{1}{2+s}k+\frac{1+s}{2+s}(N-k)\right)}= f(k/N),$$ with $f(x)=\frac{1}{\frac{1}{2+s}\frac{1}{1-x}+\frac{1+s}{2+s}\frac{1}{x}}.$ Since $f'(0+)>0$, $f(0)=0$ and $f>0$ on $(0,b]$, there exists $\epsilon>0$ such that $p_k>p_{\lfloor \epsilon N\rfloor}$ for all $\lfloor \epsilon N\rfloor\leq k\leq \lfloor  bN\rfloor$ if $N$ is large enough. Moreover, as $f'(0)=\frac{2+s}{1+s}$, $p_{\lfloor \epsilon N\rfloor}>\epsilon$ if $N$ is large enough.
Now on $T_k<\infty$, $T_k=\sum_{i=0}^{S_k-1}L_i $, with the $L_i$'s being, conditionally on $H$, independent geometric random variables with parameter $p_{H_i}^{}$.
Therefore
\begin{align*}
\mathbb{P}(T_{\lfloor bN\rfloor}>CN \,| \,S_{\lfloor \epsilon N\rfloor}\geq S_{\lfloor bN\rfloor},S_{\lfloor bN\rfloor}\leq KN)&\leq \mathbb{P}\Big(\sum_{i=0}^{KN}L_i>CN\Big)
\end{align*} and therefore
\begin{align*}
\mathbb{P}(T_{\lfloor bN\rfloor}>CN, S_{\lfloor \epsilon N\rfloor}\geq S_{\lfloor bN\rfloor},S_{\lfloor bN\rfloor}\leq KN)&\leq \mathbb{P}\Big(\sum_{i=0}^{KN}L_i>CN\Big).
\end{align*}
Now, recall that any geometric r.v. of parameter $p'_2>p'_1$ can be decomposed into the sum of $M$ independent geometric random variables of parameter $p'_1$, $M$ being an geometric independent random variable of parameter $p'_1/p'_2$. We can infer that:
\begin{align*}
\mathbb{P}(T_{\lfloor bN\rfloor}>CN, S_{\lfloor \epsilon N\rfloor}\geq S_{\lfloor bN\rfloor},S_{\lfloor bN\rfloor}\leq KN)&\leq \mathbb{P}\Big(\sum_{i=0}^{KN}D_i>CN\Big)
\end{align*}
where the $D_i$'s are independent geometric random variables with parameter $p_{\lfloor \epsilon N\rfloor}$. Since the expectation and the variance of $\sum_{i=0}^{KN}D_i$ are respectively smaller than $\frac{KN}{\epsilon}$ and $\frac{KN}{\epsilon^2}$ we conclude using Bienaym\'e-Chebyshev inequality if we take $C$ larger than $\frac{K}{\epsilon}$.
\end{proof}
Note that the previous result can be extended to the case where $Y_0/N$ converges in probability to $a\in(0,1)$.
\bigskip
\paragraph*{Ancestors genetic weights} Let us now focus on ancestors weights. Due to selection, the sequence of ancestors weights $((W_n(i,j)_{1\leq i,j\leq N})_{n\in\mathbb{Z}_+}$ is not a Markov chain. However, the couple $((W_n(i,j)_{1\leq i,j\leq N},\mathcal{Y}_n)_{n\in\mathbb{Z}_+}$ is, and satisfies the following property.

\begin{lemma}\label{prop-WY}
The sequence $(W_n,\mathcal{Y}_n)_{n\in\mathbb{Z}_+}$ is a Markov chain, with transition rules \begin{align}\mathcal{Y}_{n+1}&=\begin{cases}
  \mathcal{Y}_n\cup\{\kappa_n\} \qquad &\text{if $\mu_n\in\mathcal{Y}_n$ and  $\kappa_n\notin\mathcal{Y}_n$ } \\
  \mathcal{Y}_n\setminus\{\kappa_n\}\qquad &\text{if $\mu_n\notin\mathcal{Y}_n$ and $\kappa_n\in\mathcal{Y}_n$}\\
  \mathcal{Y}_n \qquad &\text{if $\mu_n,\kappa_n\notin\mathcal{Y}_n$ or $\mu_n,\kappa_n\in\mathcal{Y}_n$,} and
\end{cases}\label{eq-Y}\\
W_{n+1}(i,j)&=\begin{cases}W_n(i,j) \quad &\text{if $i\neq \kappa_n$}\\
    \frac{W_n(\mu_n,j)+W_n(\pi_n,j)}{2} \quad &\text{if $i=\kappa_n$,}\end{cases}\label{eq-W}\end{align} where $\mu_n$ and $\pi_n$ are drawn uniformly in $1,...,N$, and $\kappa_n$ is such that $\mathbb{P}(\kappa_n=i)=\frac{1+s\mathbf{1}_{\{i\notin\mathcal{Y}_n\}}}{Y_n+(1+s)(N-Y_n)}$ for all $i\in\{1,..,N\}$.
\end{lemma}

\begin{proof} Equation \eqref{eq-Y} is equivalent to the result stated in Proposition \ref{prop:Y} and Equation \eqref{eq-W} is a consequence of the model presented in Section \ref{sec:model} and the definition \eqref{eq:defA}.
\end{proof}

Now recall the definition of the two key quantities 
\begin{align*}
U_n&=\sum_{l\in\mathcal{Y}_n}\sum_{l'\in\mathcal{Y}_0} W_n(l,l'), \quad\text{and}\quad
V_n=\sum_{l\notin\mathcal{Y}_n}\sum_{l'\in\mathcal{Y}_0} W_n(l,l')
\end{align*} and note that $U_0=Y_0$ and $V_0=0$. From now on we consider the sequence of the rescaled triplets:
$$Z_n=\left(\frac{Y_n}{N},\frac{U_n}{N},\frac{V_n}{N}\right)\in[0,1]^3$$ and denote by $(\mathcal{F}^Z_n,n\geq0)$ its natural filtration. We also denote by $\mathcal{C}^2([0,1]^3)$ the set of real-valued twice continuously differentiable functions on $[0,1]^3$. Our first aim in this section is to present the dynamics of this stochastic process and to prove that when the population size goes to infinity, this stochastic process converges to the solution of a dynamical system which can be explicitly solved.

\medskip
First, the following proposition follows from Lemma \ref{prop-WY}: 

\begin{proposition}\label{prop-Z} The sequence $(Z_n)_{n\in\mathbb{Z}_+}=\left(\frac{Y_n}{N},\frac{U_n}{N},\frac{V_n}{N}\right)_{n\in\mathbb{Z}_+}$ (which is not Markovian) is such that
\begin{equation}\label{eq-dynamiqueZ}\begin{aligned}
Z_{n+1}=Z_n+\frac{1}{N}&\left[\left(1,\sum_{l'\in\mathcal{Y}_0}\frac{W_n(\mu_n,l')+W_n(\pi_n,l')}{2},\sum_{l'\in\mathcal{Y}_0}-W_n(\kappa_n,l')\right)\mathbf{1}_{\{\mu_n\in\mathcal{Y}_n,\kappa_n\notin\mathcal{Y}_n\}}\right.\\&+\left(-1,-\sum_{l'\in\mathcal{Y}_0}W_n(\kappa_n,l'),\sum_{l'\in\mathcal{Y}_0}\frac{W_n(\mu_n,l')+W_n(\pi_n,l')}{2}\right)\mathbf{1}_{\{\mu_n\notin\mathcal{Y}_n,\kappa_n\in\mathcal{Y}_n\}}\\&+\left(0,\sum_{l'\in\mathcal{Y}_0}\frac{W_n(\mu_n,l')+W_n(\pi_n,l')}{2}-\sum_{l'\in\mathcal{Y}_0}W_n(\kappa_n,l'),0\right)\mathbf{1}_{\{\mu_n\in\mathcal{Y}_n,\kappa_n\in\mathcal{Y}_n\}}\\&+\left.\left(0,0,\sum_{l'\in\mathcal{Y}_0}\frac{W_n(\mu_n,l')+W_n(\pi_n,l')}{2}-\sum_{l'\in\mathcal{Y}_0}W_n(\kappa_n,l')\right)\mathbf{1}_{\{\mu_n\notin\mathcal{Y}_n,\kappa_n\notin\mathcal{Y}_n\}}\right].\end{aligned}
\end{equation}
In particular (as $\sum_jW_n(i,j)\leq1$), $\|Z_{n+1}-Z_n\|\leq \sqrt{3}/N$ for all $n\geq0$, and 
\begin{equation}\label{eq-ExpDZ}\begin{aligned}\mathbb{E}\left(Z_{n+1}-Z_n|\mathcal{F}^Z_n\right)=\frac{1}{N}\Big(&\frac{sY_n/N(1-Y_n/N)}{Y_n/N+(1+s)(1-Y_n/N)},\\&\frac{U_n}{2N}+\frac{U_n+V_n}{2N}\frac{Y_n}{N}-\frac{U_n/N}{Y_n/N+(1+s)(1-Y_n/N)},\\&\frac{V_n}{2N}+\frac{U_n+V_n}{2N}\left(1-\frac{Y_n}{N}\right)-\frac{(1+s)V_n/N}{Y_n/N+(1+s)(1-Y_n/N)}\Big).\end{aligned}\end{equation}
\end{proposition}

\begin{proof} 
Equation \eqref{eq-dynamiqueZ} follows from the model described in Section \ref{sec:model} (i.e. from Lemma \ref{prop-WY}), and Equation \eqref{eq-ExpDZ} follows from \eqref{eq-dynamiqueZ} and from the fact that $\mu_n$ and $\pi_n$ are uniformly chosen in $\{1,...,N\}$ and $\kappa_n\in\mathcal{Y}_n$ with probability $\frac{Y_n}{Y_n+(1+s)(N-Y_n)}$. For example the first term $\frac{U_n}{2N}$ in the second coordinate of the right-hand side of Equation \eqref{eq-ExpDZ} is equal to $\mathbb{E}\Big(\sum_{l'\in\mathcal{Y}_0}\frac{W_n(\mu_n,l')}{2}\mathbf{1}_{\{\mu_n\in\mathcal{Y}_n\}}\Big)$ which appears when adding the first and third terms in Equation \eqref{eq-dynamiqueZ}.
\end{proof}

From the previous proposition we will derive the convergence, as the population size goes to infinity, of the stochastic sequence $(Z_{\lfloor N.\rfloor})_N$, towards a solution of the following dynamical system:

\begin{equation}\label{eq:dyn-system}\begin{cases}y'=\frac{sy(1-y)}{y+(1+s)(1-y)}\\u'=\left[\frac{u}{2}+\frac{u+v}{2}y-\frac{u}{y+(1+s)(1-y)}\right]\\
v'=\left[\frac{v}{2}+\frac{u+v}{2}(1-y)-\frac{(1+s)v}{y+(1+s)(1-y)}\right],\end{cases}\end{equation}
which is a first element of Theorem \ref{thmCvceDeterministe}.
\subsection{Dynamical system}

We first determine this solution of the system.

\begin{proposition} \label{prop-z} The differential equation \eqref{eq:dyn-system} admits a unique solution $z_t:=(y_t,u_t,v_t)_{t\geq0}$ starting from $(a,a,0)$ with $a\in(0,1)$. This solution satisfies:
\begin{equation}\label{eq-sol-syst}\left\{\begin{aligned}
      y_t&= F^{-1}\left(\frac{a^{1+s}}{1-a}\exp(st)\right) \qquad\text{where $F:x\rightarrow\frac{x^{1+s}}{1-x}$ maps $[0,1)$ onto $[0,\infty)$}\\
      u_t&= y_t\frac{a^{\frac{1+s}{2s}}}{(1-a)^{\frac{1}{2s}}}\left[\frac{(1-y_t)^{\frac{1}{2s}}}{y_t^{\frac{1+s}{2s}}}+\int_a^{y_t}\frac{(1-x)^{\frac{1}{2s}}}{x^{\frac{1+s}{2s}}}\left[\frac{1}{2}+\frac{1}{2s}\frac{1}{1-x}\right]dx\right] \\
      v_t&= (1-y_t)\frac{a^{\frac{1+s}{2s}}}{(1-a)^{\frac{1}{2s}}}\int_a^{y_t}\frac{(1-x)^{\frac{1}{2s}}}{x^{\frac{1+s}{2s}}}\left[\frac{1}{2}+\frac{1}{2s}\frac{1}{1-x}\right]dx. \\
    \end{aligned}\right.\end{equation}
\end{proposition}

\begin{proof} Let us first focus on the first differential equation of \eqref{eq:dyn-system} which happens to be autonomous and easily solvable: \begin{equation}\label{eq-dydt}\frac{dy_t}{dt}=\frac{sy_t(1-y_t)}{y_t+(1+s)(1-y_t)}=\frac{sy_t(1-y_t)}{1+s(1-y_t)}=:m(y_t),\end{equation} with $y_0=a$. Since the function $m$ is Lipschitz continuous on $[0,1]$, from Picard-Lindel\"{o}f (or Cauchy-Lipschitz) Theorem,
 this ordinary differental equation admits a unique solution starting at $a\in[0,1]$. Since $m(0)=m(1)=0$, this solution is simply constant if $a\in\{0,1\}$, and if $a\in(0,1)$, then this solution is such that $y_t\in(0,1)$ for all $t$, and the differential equation \eqref{eq-dydt} can then also be written as:
$$\frac{dy_t}{dt}\left[\frac{1+s}{y_t}+\frac{1}{1-y_t}\right]=s.$$ Noting that $\frac{1+s}{y}$ and $\frac{1}{1-y}$ are respectively the derivative of $y\rightarrow(1+s)\ln(y)$ and $y\rightarrow-\ln(1-y)$ gives that the solution to this equation starting from $a$ satisfies 
\begin{equation}\label{eq-y^1+s}\frac{y_t^{1+s}}{1-y_t}=\frac{a^{1+s}}{1-a}\exp(st)\in[0,\infty) \quad\text{for all $t\geq0$}.\end{equation} Let $F$ be the function such that $F(x)=\frac{x^{1+s}}{1-x}$ for all $x\in[0,1)$. The function $F$ is strictly increasing on $[0,1)$ and sends $[0,1)$ onto $[0,+\infty)$. Therefore, since $a\in[0,1)$, Equation \eqref{eq-y^1+s} gives that
$y_t= F^{-1}\left(\frac{a^{1+s}}{1-a}\exp(st)\right)$ for all $t\geq0$, which is the first equation of \eqref{eq-sol-syst}. 

\medskip
Note that it gives in particular that 
\begin{align}\label{eq-yt}
\exp(t/2)=\exp(st)^{\frac{1}{2s}}=\left(F(y_t)\frac{1-a}{a^{1+s}}\right)^{\frac{1}{2s}}=\frac{y_t^{\frac{1+s}{2s}}}{(1-y_t)^{\frac{1}{2s}}}\frac{(1-a)^{\frac{1}{2s}}}{a^{\frac{1+s}{2s}}}.
\end{align}
This formula will be used at the end of the proof.
Let us now come back to the whole differential equation \eqref{eq:dyn-system}. The existence and uniqueness of the solution starting from the initial condition $(a,a,0)$ falls again from Cauchy-Lipschitz theorem. This solution $(z_t)_{t\geq0}$ takes values in $[0,1]^3$. Now from Equation \eqref{eq:dyn-system},
\begin{align*}
    \left(\frac{u}{y}\right)'&=\frac{u'}{y}-u\frac{y'}{y^2}=\frac{u}{2y}+\frac{u+v}{2}-\frac{u}{y[y+(1+s)(1-y)]}-\frac{usy(1-y)}{y^2[y+(1+s)(1-y)]}\\&=\frac{u+v}{2}-\frac{u}{2y},
\end{align*}

and similarly,
\begin{equation}\label{eqcalculderivC}
\begin{aligned}
    \left(\frac{v}{1-y}\right)'&=\frac{v'}{1-y}+v\frac{y'}{(1-y)^2}\\&=\frac{v}{2(1-y)}+\frac{u+v}{2}-\frac{(1+s) v}{(1-y)[y+(1+s)(1-y)]}+\frac{vsy(1-y)}{(1-y)^2[y+(1+s)(1-y)]}\\&=\frac{u+v}{2}-\frac{v}{2(1-y)}.
\end{aligned}
\end{equation}

Hence
\begin{align*}
    \left(\frac{u}{y}-\frac{v}{1-y}\right)'&=-\frac{1}{2}\left(\frac{u}{y}-\frac{v}{1-y}\right),
\end{align*}
which gives that 
\begin{equation*}\frac{u_t}{y_t}-\frac{v_t}{1-y_t}=\exp(-t/2)\end{equation*} since $u_0=y_0$ and $v_0=0$.

Therefore considering the quantity $D_t=\frac{u_te^{t/2}}{y_t}-\frac{v_te^{t/2}}{1-y_t}$, one has that \begin{equation}\label{EquationD}\frac{dD_t}{dt}=0 \qquad \text{hence}\qquad D_t=D_0=1 \qquad\text{for all $t\geq0$.}\end{equation} Note that this property is a fundamental characteristic of our model. An analogous version of it was given in \cite{CoronLeJan22} (Equation (1.7)) in a discrete setting and for infinite selection, and can be extended to the finite selection case (see Proposition 2.6 of the unpublished note \cite{coronLeJan2024}). 

Equation \eqref{EquationD} gives that \begin{equation}\label{equationu}u_te^{t/2}=y_t\left(1+\frac{v_t}{1-y_t}e^{t/2}\right)\end{equation} for all $t\geq0$. Finally, let us set \begin{equation}\label{eqdefbeta}\beta_t=\frac{v_t}{1-y_t}e^{t/2}.\end{equation} We can express $\beta_t$ as a function of $y_t$. Indeed, since $\frac{d\beta_t}{dt}=\frac{u_t+v_t}{2}e^{t/2}$ from Equation \eqref{eqcalculderivC},

\begin{align*}\frac{d\beta_t}{dy_t}=\frac{\frac{d\beta_t}{dt}}{\frac{dy_t}{dt}}&=\frac{\frac{u_t+v_t}{2}e^{t/2}}{\frac{sy_t(1-y_t)}{y_t+(1+s)(1-y_t)}} \\&=\frac{y_t\left[1+\frac{v_t}{1-y_t}e^{t/2}\right]+v_te^{t/2}}{\frac{2sy_t(1-y_t)}{y_t+(1+s)(1-y_t)}}\quad\text{from Equation \eqref{equationu}} \\&=\frac{y_t+\beta_t}{\frac{2sy_t(1-y_t)}{y_t+(1+s)(1-y_t)}}.\end{align*}
So \begin{equation}\label{eq-diff-beta}\frac{d\beta_t}{dy_t}=\frac{\beta_t}{\frac{2sy_t(1-y_t)}{y_t+(1+s)(1-y_t)}}+\frac{y_t+(1+s)(1-y_t)}{2s(1-y_t)}\end{equation} 
and this linear autonomous equation can be solved by variation of parameters: let us introduce a function $C$ defined on $(0,1)$ such that

\begin{align*}\beta_t=C(y_t)\frac{y_t^{\frac{1+s}{2s}}}{(1-y_t)^{\frac{1}{2s}}}\quad\text{for all $t\geq0$}.\end{align*}
Then from Equation \eqref{eq-diff-beta},
\begin{align*}\frac{dC}{dy}\frac{y^{\frac{1+s}{2s}}}{(1-y)^{\frac{1}{2s}}}=\frac{y+(1+s)(1-y)}{2s(1-y)}=\frac{y}{2s(1-y)}+\frac{1+s}{2s},\end{align*} therefore
\begin{align*}\frac{dC}{dy}=\frac{(1-y)^{\frac{1}{2s}-1}}{2sy^{\frac{1+s}{2s}-1}}+\frac{1+s}{2s}\frac{(1-y)^{\frac{1}{2s}}}{y^{\frac{1+s}{2s}}}.\end{align*}

Hence since $\beta_0=0$ and $y_0=a$, 
\begin{align*}C(y)&=\frac{1}{2s}\int_a^{y}\frac{(1-x)^{\frac{1}{2s}}}{x^{\frac{1+s}{2s}}}\left[\frac{x}{1-x}+1+s\right]dx\\&=\frac{1}{2s}\int_a^{y}\frac{(1-x)^{\frac{1}{2s}}}{x^{\frac{1+s}{2s}}}\left[\frac{1}{1-x}+s\right]dx\\&=\int_a^{y}\frac{(1-x)^{\frac{1}{2s}}}{x^{\frac{1+s}{2s}}}\left[\frac{1}{2}+\frac{1}{2s}\frac{1}{1-x}\right]dx,\end{align*}

which gives that $$\beta_t=\frac{y_t^{\frac{1+s}{2s}}}{(1-y_t)^{\frac{1}{2s}}}\int_a^{y_t}\frac{(1-x)^{\frac{1}{2s}}}{x^{\frac{1+s}{2s}}}\left[\frac{1}{2}+\frac{1}{2s}\frac{1}{1-x}\right]dx.$$
Hence from Equation \eqref{eq-yt} and the definition of $\beta_t$ in \eqref{eqdefbeta}, \begin{align*}\frac{v_t}{1-y_t}&=e^{-\frac{t}{2}}\beta_t=\frac{a^{\frac{1+s}{2s}}}{(1-a)^{\frac{1}{2s}}}\int_a^{y_t}\frac{(1-x)^{\frac{1}{2s}}}{x^{\frac{1+s}{2s}}}\left[\frac{1}{2}+\frac{1}{2s}\frac{1}{1-x}\right]dx\end{align*} which gives the third equation of \eqref{eq-sol-syst}.

Finally from Equation \eqref{EquationD}, 
\begin{align*}\frac{u_t}{y_t}&=e^{-\frac{t}{2}}(1+\beta_t)=\frac{a^{\frac{1+s}{2s}}}{(1-a)^{\frac{1}{2s}}}\left[\frac{(1-y_t)^{\frac{1}{2s}}}{y_t^{\frac{1+s}{2s}}}+\int_a^{y_t}\frac{(1-x)^{\frac{1}{2s}}}{x^{\frac{1+s}{2s}}}\left[\frac{1}{2}+\frac{1}{2s}\frac{1}{1-x}\right]dx\right]\end{align*} which gives the second equation of \eqref{eq-sol-syst}.
\end{proof}

Recall that in the limiting solution $(z_t)_{t\geq0}=(y_t,u_t,v_t)_{t\geq0}$, $y_t$ is strictly increasing, from $y_0$ to $1$ (the function $y$ was used as a time change in the previous proof). For any $b\in[a,1)$, let us denote by $r_b$ the hitting time of $b$ by $(y_t)_{t\geq0}$. The limiting dynamical system $(z_t)_{t\geq0}$ satisfies the following:
\begin{corollary} \label{cor-zb}
\begin{description}
\item[(i)] For any $b\in[y_0,1]$, \begin{align*}z_{r_b}=\Big(b,b\frac{a^{\frac{1+s}{2s}}}{(1-a)^{\frac{1}{2s}}}&\left[\frac{(1-b)^{\frac{1}{2s}}}{b^{\frac{1+s}{2s}}}+\int_a^{b}\frac{(1-x)^{\frac{1}{2s}}}{x^{\frac{1+s}{2s}}}\left[\frac{1}{2}+\frac{1}{2s}\frac{1}{1-x}\right]dx\right],\\&(1-b)\frac{a^{\frac{1+s}{2s}}}{(1-a)^{\frac{1}{2s}}}\int_a^{b}\frac{(1-x)^{\frac{1}{2s}}}{x^{\frac{1+s}{2s}}}\left[\frac{1}{2}+\frac{1}{2s}\frac{1}{1-x}\right]dx\Big).\end{align*}
\item[(ii)] $(z_t)_{t\geq0}$, converges when $t$ goes to infinity, to $$z_{\infty}:=\left(1,\frac{a^{\frac{1+s}{2s}}}{(1-a)^{\frac{1}{2s}}}\left[\int_a^1\frac{(1-x)^{\frac{1}{2s}}}{x^{\frac{1+s}{2s}}}\left[\frac{1}{2}+\frac{1}{2s}\frac{1}{1-x}\right]dx\right],0\right).$$
\end{description}
\end{corollary}

\begin{proof}
 $(i)$ follows from Equation \eqref{eq-sol-syst}, with $t=r_b$. The convergence stated in $(ii)$ is immediate from Equation \eqref{eq-sol-syst} and notably from the (already mentioned) fact that $(y_t)_{t\geq0}$ is strictly increasing and converges to $1$ when $t$ goes to infinity. Note that the singularity in $1$ is integrable. 
\end{proof}

\subsection{Convergence}
We finally prove that if $Y_0=\lfloor aN\rfloor$, then for any constant $c>0$, the stochastic process $(Z_{\lfloor Nt\rfloor})_{0\leq t\leq c}$ converges, when $N$ goes to infinity, to the solution $(y_t,u_t,v_t)_{0\leq t \leq c}$ of the dynamical system \eqref{eq:dyn-system}, starting from $(a,a,0)$. Recall that we denoted by $(\mathcal{F}^Z_n,n\geq 0)$ the filtration associated to the stochastic process $(Z_n)_{n\geq 0}$.

\begin{proposition}\label{propCvce} Let $a\in(0,1)$. If $Y_0=\lfloor aN\rfloor$, then for any $c\in\mathbb{R}_+$, \begin{align}
    \sup_{0\leq t\leq c}\|Z_{\lfloor Nt\rfloor}-z_t\|\underset{N\rightarrow\infty}{\longrightarrow}0
\end{align} in probability.
\end{proposition}
\begin{proof}
From Proposition \ref{prop-Z}, for any $n\in\mathbb{Z}_+$, we obtain that: $$Z_{n+1}-Z_n=A_{n+1}+\frac{1}{N}g(Z_n)$$ where $g(Z_n)=\mathbb{E}(Z_{n+1}-Z_n|\mathcal{F}^Z_n)$ is such that \begin{align*}
    g(y,u,v)=\Big(\frac{y(1-y)s}{y+(1+s)(1-y)},&\frac{u}{2}+\frac{u+v}{2}y-\frac{u}{y+(1+s)(1-y)},\\&\frac{v}{2}+\frac{u+v}{2}(1-y)-\frac{(1+s)v}{y+(1+s)(1-y)}\Big)\end{align*} for all $(y,u,v)\in[0,1]^3$, and $A_{n+1}=Z_{n+1}-Z_n-\mathbb{E}(Z_{n+1}-Z_n|\mathcal{F}^Z_n)=Z_{n+1}-\mathbb{E}(Z_{n+1}|\mathcal{F}^Z_n)$. Therefore
\begin{align}\label{DecompoZ}
    Z_n=Z_0+\sum_{k=1}^n A_k+\frac{1}{N}\sum_{k=0}^{n-1} g(Z_k)
\end{align}
and the variables $(A_k)_{1\leq k\leq n}$
are the increments of the $\mathcal{F}^Z$-martingale $(M_n)_{n\geq 0}$ such that $M_n=\sum_{k=1}^nA_k$ for any $n\geq 0$. 
Now from Doob's martingale inequality, for all $c>0$, 
$$\mathbb{E}\left(\sup_{0\leq n\leq \lfloor c N\rfloor} (M_n)^2\right)\leq  4\mathbb{E}\left(M_{\lfloor cN\rfloor}^2\right).$$ 
Besides, from the martingale property, $$\mathbb{E}(M_{\lfloor cN\rfloor}^2)=\mathbb{E}\left(\Big(\sum_{k=1}^{\lfloor cN\rfloor}A_k\Big)^2\right)=\mathbb{E}\left(\sum_{k=1}^{\lfloor cN\rfloor}(A_k)^2\right)\leq \frac{12c}{N}\quad\text{as $\|Z_{n+1}-Z_n\|^2\leq \frac{3}{N^2}$,}$$ which gives that 
$$\sup_{0\leq t\leq c}\left\|\sum_{k=1}^{\lfloor Nt\rfloor} A_k \right\|\longrightarrow 0 \quad\quad \text{in probability when $N$ goes to infinity.}$$ 
Now from Proposition \ref{prop-z},
\begin{align}
    z_t=(a,a,0)+\int_{0}^{t}g(z_s)ds
\end{align} for all $t\in[0,c]$.
Therefore by \eqref{DecompoZ},
\begin{align}
    Z_{\lfloor Nt\rfloor}-z_t=(Y_0/N,Y_0/N,0)-(a,a,0)+\sum_{k=1}^{\lfloor Nt\rfloor}A_k+\int_0^t g(Z_{\lfloor Ns\rfloor})-g(z_s)ds + R_N(t)
\end{align}
where $|R_N(t)|\leq \frac{\|g\|_{\infty}}{N}$.
Now since the Jacobian of the function $g$ is bounded, there exists $K\in\mathbb{R}_+$ such that $\|g(z)-g(z')\|\leq K\|z-z'\|$ for all $z,z'\in[0,1]^3$. Hence
\begin{align*}
    \left\|Z_{\lfloor Nt\rfloor}-z_t\right\|\leq\|(Y_0/N,Y_0/N,0)-(a,a,0)\|+\left\|\sum_{k=1}^{\lfloor Nt\rfloor}A_k\right\|+\int_0^t K\|Z_{\lfloor Ns\rfloor}-z_s)\|ds+\frac{\|g\|_{\infty}}{N}.
\end{align*} 
Therefore by Gronwall's inequality, for any $t\in[0,c]$,
\begin{align*}
    \|Z_{\lfloor Nt\rfloor}-z_t\|\leq\left(\|(Y_0/N,Y_0/N,0)-(a,a,0)\|+\left\|\sum_{k=1}^{\lfloor Nt\rfloor}A_k\right\|+\frac{\|g\|_{\infty}}{N}\right)e^{Kt}
\end{align*} which gives the result.
\end{proof}

Corollary \ref{cor_main} follows directly from Theorem \ref{thmCvceDeterministe}, Proposition \ref{prop-Y-TN} (ii) and Corollary \ref{cor-zb} (i).

\bibliographystyle{plain}
\bibliography{biblio}

\begin{thebibliography}{10}

\bibitem{BaakeLenzWakolbinger2016}
Ellen Baake, Ute Lenz, and Anton Wakolbinger.
\newblock The common ancestor type distribution of a {$\Lambda$}-wright-fisher
  process with selection and mutation.
\newblock {\em Electronic Communications in Probability}, 21, 03 2016.

\bibitem{BartonEtheridge2011}
Nicholas~H. Barton and Alison~M. Etheridge.
\newblock The relation between reproductive value and genetic contribution.
\newblock {\em Genetics}, 188(4):953--973, 2011.

\bibitem{Chang1999}
Joseph~T. Chang.
\newblock Recent common ancestors of all present-day individuals.
\newblock {\em Advances in Applied Probability}, 31(4):1002--1026, 1999.

\bibitem{Cordero2017}
Fernando Cordero.
\newblock The deterministic limit of the {M}oran model: a uniform central limit
  theorem.
\newblock {\em Markov Processes and Related Fields}, 23:313--324, 07 2017.

\bibitem{geneal}
Camille Coron and Yves Le~Jan.
\newblock Pedigree in the biparental {M}oran model.
\newblock {\em Journal of Mathematical Biology}, 84(51), 2022.

\bibitem{CoronLeJan22}
Camille Coron and Yves Le~Jan.
\newblock Genetic contribution of an advantaged mutant in the biparental
  {M}oran model.
\newblock {\em Ukrainian Mathematical Journal}, 2024.

\bibitem{coronLeJan2024}
Camille Coron and Yves Le~Jan.
\newblock Genetic contribution of an advantaged mutant in the biparental
  {M}oran model -- finite selection.
\newblock {\em ArXiv 2405.08404}, 2024.

\bibitem{Derrida2000}
Bernard Derrida, Susanna~C. Manrubia, and Damian~H. Zanette.
\newblock On the genealogy of a population of biparental individuals.
\newblock {\em Journal of Theoretical Biology}, 203(3):303--315, 2000.

\bibitem{Etheridgeetal2006}
Alison Etheridge, Peter Pfaffelhuber, and Anton Wakolbinger.
\newblock An approximate sampling formula under genetic hitchhiking.
\newblock {\em The Annals of Applied Probability}, 16(2):685--729, 2006.

\bibitem{EtheridgeGriffiths}
Alison~M. Etheridge and Robert~C. Griffiths.
\newblock A coalescent dual process in a {M}oran model with genic selection.
\newblock {\em Theoretical Population Biology}, 75(4):320--330, 2009.

\bibitem{Fernhead2002}
Paul Fearnhead.
\newblock The common ancestor at a nonneutral locus.
\newblock {\em Journal of Applied Probability}, 39(1):38--54, 2002.

\bibitem{feller-vol-1}
William Feller.
\newblock {\em {An Introduction To Probability Theory And Its Applications.
  {V}ol. {I}}}.
\newblock Third edition. John Wiley \& Sons Inc., New York, 1968.

\bibitem{FortesLimaetal2021}
Cesar~A. Fortes-Lima, Romain Laurent, Valentin Thouzeau, Bruno Toupance, and
  Paul Verdu.
\newblock Complex genetic admixture histories reconstructed with approximate
  bayesian computation.
\newblock {\em Molecular Ecology Resources}, 21(4):1098--1117, 2021.

\bibitem{GravelSteel2015}
S.~Gravel and M.~Steel.
\newblock The existence and abundance of ghost ancestors in biparental
  populations.
\newblock {\em Theoretical Population Biology}, 101:47--53, 2015.

\bibitem{Kimura1964}
Motoo Kimura.
\newblock Diffusion models in population genetics.
\newblock {\em Journal of Applied Probability}, 1(2):177--232, 1964.

\bibitem{Kluthetal2013}
S.~Kluth, T.~Hustedt, and E.~Baake.
\newblock The common ancestor process revisited.
\newblock {\em bulletin of mathematical biology}, 75:2003--2027, 2013.

\bibitem{KluthBaake}
Sandra Kluth and Ellen Baake.
\newblock The {M}oran model with selection: Fixation probabilities, ancestral
  lines, and an alternative particle representation.
\newblock {\em Theoretical population biology}, 90, 09 2013.

\bibitem{KroneNeuhauser1997}
Stephen~M Krone and Claudia Neuhauser.
\newblock Ancestral processes with selection.
\newblock {\em Theoretical Population Biology}, 51(3):210--237, 1997.

\bibitem{Lambertetal2018}
Amaury Lambert, Ver\'onica Mir\'o~Pina, and Emmanuel Schertzer.
\newblock Chromosome painting: How recombination mixes ancestral colors.
\newblock {\em The Annals of Applied Probability}, 31(2):826--864, 2021.

\bibitem{Lenzetal2015}
Ute Lenz, Sandra Kluth, Ellen Baake, and Anton Wakolbinger.
\newblock Looking down in the ancestral selection graph: A probabilistic
  approach to the common ancestor type distribution.
\newblock {\em Theoretical Population Biology}, 103:27--37, 2015.

\bibitem{MatsenEvans2008}
Frederick~A. Matsen and Steven~A. Evans.
\newblock To what extent does genealogical ancestry imply genetic ancestry?
\newblock {\em Theoretical Population Biology}, 2008.

\bibitem{PfaffelhuberWakolbinger2023}
Peter Pfaffelhuber and Anton Wakolbinger.
\newblock A diploid population model for copy number variation of genetic
  elements.
\newblock {\em Electronic Journal of Probability}, 28, 2023.

\bibitem{Durrett2004}
Jason Schweinsberg and Rick Durrett.
\newblock {Random partitions approximating the coalescence of lineages during a
  selective sweep}.
\newblock {\em The Annals of Applied Probability}, 15(3):1591 -- 1651, 2005.

\bibitem{Stephanetal2006}
Wolfgang Stephan, Yun~S Song, and Charles~H Langley.
\newblock The hitchhiking effect on linkage disequilibrium between linked
  neutral loci.
\newblock {\em Genetics}, 172(4):2647--2663, 2006.

\bibitem{Taylor2007}
Jesse Taylor.
\newblock {The Common Ancestor Process for a Wright-Fisher Diffusion}.
\newblock {\em Electronic Journal of Probability}, 12:808--847, 2007.

\bibitem{VerduRosenberg2011}
Paul Verdu and Noah~A Rosenberg.
\newblock A general mechanistic model for admixture histories of hybrid
  populations.
\newblock {\em Genetics}, 189(4):1413--1426, 12 2011.

\bibitem{WakeleyKingWilton2016}
John Wakeley, L\'{e}andra King, and Peter~R. Wilton.
\newblock Effects of the population pedigree on genetic signatures of
  historical demographic events.
\newblock {\em Proceedings of the National Academy of Sciences},
  113(29):7994--8001, 2016.

\end{thebibliography}

\end{document}